\documentclass[a4paper,10pt]{amsart}

\usepackage{amssymb,amsmath,amsfonts,amsthm,mathtools}
\usepackage{latexsym,mathrsfs,pb-diagram}
\usepackage[pdftex]{graphicx}
\usepackage[all]{xy}
\usepackage[hmargin=2.5cm,vmargin=3.5cm]{geometry}
\usepackage[plainpages=false,colorlinks,pdfpagelabels]{hyperref}
\hypersetup{
  urlcolor=black,
  citecolor=green,
  linkcolor=blue
}

\numberwithin{equation}{section}

\theoremstyle{definition}
\newtheorem{defn}{Definition}[section]

\theoremstyle{remark}

\newtheorem{rmk}[defn]{Remark}

\theoremstyle{plain}
\newtheorem{prop}[defn]{Proposition}
\newtheorem{cor}[defn]{Corollary}
\newtheorem{thm}[defn]{Theorem}
\newtheorem{lem}[defn]{Lemma}

\newcommand{\Z}{\mathbb{Z}}
\newcommand{\R}{\mathbb{R}}

\newcommand{\C}{\mathbb{C}}

\newcommand{\A}{\mathcal{A}}
\newcommand{\G}{\mathrm{G}}

\newcommand{\N}{\mathrm{N}}
\newcommand{\m}{\mathrm{m}}
\newcommand{\F}{\mathcal{F}}
\newcommand{\M}{\mathcal{M}}
\newcommand{\T}{\mathrm{T}}

\renewcommand{\>}{\right\rangle}
\renewcommand{\L}{\mathcal{L}}
\renewcommand{\epsilon}{\varepsilon}
\renewcommand{\d}{\mathrm{d}}
\newcommand{\Ci}{{\mathcal{C}}^{\infty}}

\newcommand{\Cm}{\mathcal{C}^\mathcal{M}}
\newcommand{\WF}{\mathrm{WF}}

\newcommand{\Char}{\mathrm{Char}}
\renewcommand{\Re}{\mathrm{Re}}
\renewcommand{\Im}{\mathrm{Im}}

\author{Nicholas Braun Rodrigues}
\address{University of S{\~a}o Paulo, IME-USP, S{\~a}o Paulo, SP, Brazil}
\email{braun@ime.usp.br}
\author{Antonio V. da Silva Jr.}
\address{University of S{\~a}o Paulo, IME-USP, S{\~a}o Paulo, SP, Brazil}
\email{avictor@ime.usp.br}
\thanks{This work was supported by CNPq, Conselho Nacional de Desenvolvimento Cient{\'i}fico e Tecnol{\'o}gico.}

\title [Approximate solutions of vector fields and Denjoy-Carleman regularity]
{Approximate solutions of vector fields and an application to Denjoy-Carleman regularity of solutions of a nonlinear PDE}
\begin{document}

\begin{abstract} 
In this paper we study microlocal regularity of a $\mathcal{C}^2$ solution $u$ of the equation
\begin{equation*}
u_t = f(x,t,u,u_x),
\end{equation*}
where $f(x,t,\zeta_0, \zeta)$ is ultradifferentiable  in the variables $(x,t)\in \R^{N} \times \R$ and holomorphic in the variables $(\zeta_0,\zeta) \in \C \times \C^{N}$. 
We proved that if $\Cm$ is a regular Denjoy-Carleman class (including the quasianalytic case) then:
\begin{equation*}
\WF_\M (u)\subset \Char(L^u),
\end{equation*}
where $\WF_\M(u)$ is the Denjoy-Carleman wave-front set of $u$ and $\Char(L^u)$ is the characteristic set of the linearized operator $L^u$: 
\begin{equation*}
L^u = \dfrac{\partial}{\partial t} - \sum_{j=1}^{N}\frac{\partial f}{\partial\zeta_j}(x,t,u,u_x)\dfrac{\partial}{\partial x_j}.
\end{equation*}

\end{abstract}

\maketitle

\section{Introduction}

Let $\Omega^\prime\subset \R^N \times \R$ and $\Omega^{\prime\prime} \subset \C \times \C^N$ be open sets and let $f \in \Ci(\Omega^\prime\times\Omega^{\prime\prime})$ be holomorphic with respect to the variables $(\zeta_0,\zeta) \in \C \times \C^N$. 
Suppose that $u\in \mathcal{C}^2(\Omega^\prime)$ is a solution of the nonlinear equation:
\begin{equation*}
u_t = f(x,t,u,u_x),
\end{equation*}
and consider the linearized operator:
\begin{equation*}
L^u = \dfrac{\partial}{\partial t} - \sum_{j=1}^N\frac{\partial f}{\partial \zeta_j}(x,t,u,u_x)\dfrac{\partial}{\partial x_j}.
\end{equation*}
Many authors have studied the relation between the microlocal regularity of $u$ and the characteristic set of the linearized operator $L^u$ for different assumptions on the regularity of the function $f$ in the variables $(x,t)$. 
In \cite{treves:92} F. Treves and N. Hanges proved that if $f$ is real-analytic in $(x,t)$ then the real-analytic wave front set of $u$ is contained in the characteristic set of $L^u$. 
The $\Ci$ version of this result is a consequence of a result proved by J. Y. Chemin in \cite{chemin:88}, a different proof of it being obtaind by C. H. Asano, in \cite{asano:95}, by adapting Hanges-Treves' techniques.
Later on, R. F. Barostichi and G. Petronilho proved in \cite{petronilho:09} that if $f$ is Gevrey in $(x,t)$ then the same result is valid for the Gevrey wave-front set. 
Finally, Z. Adwan and G. Hoepfner proved in \cite{hoepfner:10} analogous results for strongly non-quasianalytic Denjoy-Carleman classes. 
The main difference between Asano's and Treves-Hanges' proofs is the availability of Cauchy-Kowalevski in the analytic setting while in the $\Ci$ case the proof relies on approximate solutions of vector fields and almost-analytic extensions. 
The main difficulty in the Gevrey and in the strongly non-quasianalytic case is to find a suitable approximate solution that belongs to the class under consideration. 

In this work we deal with the same problem as in \cite{treves:92}, \cite{asano:95}, \cite{petronilho:09} and \cite{hoepfner:10}, but in the case of regular Denjoy-Carleman classes. 
The only extra hypothesis that we make is that the space of the real-analytic functions is properly contained in the Denjoy-Caleman class under consideration. 
This includes the quasi-analytic case, and in that case, we gain a difficulty: the absence of non-trivial flat functions. 
This is an obstruction for the technique that Asano, Barostichi-Petronilho and Adwan-Hoepfner used in their proofs. 

Loosely speaking if $u_0$ is a function ($\Ci$, Denjoy-Carleman, Gevrey) in an open set $\Omega$, and $L$ is a vector field in $\Omega\times [-1,1]$, a function $u$ on $\Omega\times[-1,1]$ is an approximate solution of $L$ with initial datum $u_0$ if $u(x,0)=u_0(x)$ and $Lu$ is $t$-flat at $t=0$. 
If our class is quasi-analytic and the approximate solution $u$ belongs to this class, we would have that $Lu$ is actually zero. 
So finding approximate solutions in this case is as difficult as finding solutions of the Cauchy problem with initial datum $u_0$. 
To circumvent this difficulty we have to be able to leave the quasianalytic class, more precisely, we have to construct a suitable approximate solution $u$ that is only a $\Ci$-function and we need a more precise notion of $t$-flatness. 
In fact, in \cite{petronilho:09} and \cite{hoepfner:10} this notion is already used. 
Let $\Cm$ be the Denjoy-Carleman class associated with the sequence $\M=(M_k)_{k=0}^\infty$. 
We say that $Lu$ is $(\M,t)$-flat if 
\[
|Lu(x,t)|\leq \frac{C^{k+1}M_k}{k!}|t|^k, \quad \forall k\in\Z_+,
\]
for some constant $C > 0$. 
The difference here is that if $u$ is in the same class of $u_0$, then so is $Lu$, and using Taylor's formula one obtain the inequality above. 
So the difficult part is to prove the existence of a $\Ci$-approximate solution $u$ such that $Lu$ is $(\M,t)$-flat.
We construct such approximate solution by adapting an extension theorem due to E. M. Dyn'kin presented in \cite{dyn'kin}. 
In that paper Dyn'kin proved that given a $\Cm$-function on an open set $\Omega \subset \R^N$ there exists a suitable almost-analytic extension of $u$ in the complex space, \textit{i.e.} there exists a function $U \in \Ci(\C^N)$ such that $U(x) = u(x)$, $\forall x \in \Omega$, and 
\begin{equation*}
\left|\frac{\partial U}{\partial\bar{z}_j}(z)\right|\leq \frac{C^{k+1}M_k}{k!}|\Im \, z|^k, \qquad k\in\Z_+, \quad j=1,2,\dots,N.
\end{equation*}
Stated differently, $U$ is an $(\M,|\Im \, z|)$-approximate solution for the complex $\{\partial/\partial{\bar{z}_j}\}_{j=1}^N$. 
In this paper we adapt Dyn'kin's proof for the case of a vector field of the form
\begin{equation*}
L = \dfrac{\partial}{\partial t} + \sum_{j=1}^Na_j(x,t)\dfrac{\partial}{\partial x_j}.
\end{equation*}
With this in our hands and other results concerning general Denjoy-Carleman functions, such as the characterization of the Denjoy-Carleman wave-front set given by the FBI-transform, we can prove the Hanges-Treves result for general regular Denjoy-Carleman classes.

We organize the paper as follows: in Section \ref{sec:denjoy} we state and prove some results about regular Denjoy-Carleman classes following \cite{dyn'kin}, in Section \ref{sec:approx} we prove the theorem about approximate solutions, Theorem \ref{thm:approx_sol}, and finally in the Section \ref{sec:nonlinear} we use Theorem \ref{thm:approx_sol} to prove the main result of this paper, Theorem  \ref{thm:sol_approx_linearized}, and then applying the same argument of Hanges-Treves in \cite{treves:92} we prove the desired regularity result.

We wish to thank professor P. D. Cordaro for presenting us this problem. We are also grateful to P. D. Cordaro, G. Ara\'ujo and O. Erazo for the fruitful discussion in our weekly seminar, to L. F. Ragognette for his suggestions and to S. F\"{u}rd\"{o}s for his comments during his visit to Brazil.

\section{Denjoy-Carleman classes}
\label{sec:denjoy}
In this section we recall the definitions and some properties of the regular Denjoy-Carleman classes as defined in \cite{dyn'kin}.\\

Let $\mathcal{M} = (M_k)_{k=0}^\infty$ be a sequence of positive real numbers. 
We say that $\mathcal{M}$ is regular if the sequence $(m_k)_{k=0}^\infty$, where $m_k=M_k/k!$, has the following properties:
\begin{enumerate}
\item[a)] $m_0=m_1=1$;
\item[b)] $m_k^2 \leq m_{k-1}m_{k+1},\quad k\geq 1$;
\item[c)] $\sup \left(m_{k+1}/m_k\right)^{1/k}<\infty$;
\item[d)] $\lim_{k\to\infty}m_k^{1/k}=\infty$.
\end{enumerate} 
The conditions a) and b) imply that the sequence $m_k$ is increasing; 
condition c) gives us a constant $c>0$ such that $m_{k+1}\leq c^km_k$, for all $k=0,1,2,\dots$; 
condition b) is often called strong $\log$-convexity. 
If $\Omega \subset \R^N$ is an open set, the space $\Cm(\Omega)$ of ultradifferentiable functions associated to the regular sequence $\M$ is the space of all $\Ci$-functions $f$ such that for every compact $K \subset \Omega$ there is a positive constant $A$ for which the following inequality holds:
\begin{equation*}
\sup_{x\in K}|D^\alpha f(x)|\leq A^{|\alpha|+1}M_{\alpha},\quad \forall \alpha \in \Z_+^N.
\end{equation*}
Thus, setting $M_k = k!^s$, $s > 1$, one obtain the Gevrey classes $\G^s$.
As in \cite{bch} we define the FBI transform of a compactly supported distribution $u$ by
\begin{equation*}
\F[u](x,\xi)=u_y\left(e^{i(x-y)\cdot\xi-|\xi|(x-y)^2}\right).
\end{equation*}
In \cite{furdos} it is proved that a compactly supported distribution $u$ belongs to $\Cm$ if and only if for every compact $K$ there is a positive constant $A$ such that:
\begin{equation*}
|\F[u](x,\xi)|\leq \frac{A^{k+1}M_k}{|\xi|^k}, \quad k \in \Z_+, \; x \in K, \; \xi \in \R^N.
\end{equation*}
This last inequality can be used to microlocalize the notion of $\Cm$ regularity, thus we can define the Denjoy-Carleman wave-front set of a distibution $u$ at a point $x$, denoted by $\WF_\M(u)|_x$, as the complementary set of the $\Cm$-regular directions.
Now we will recall some functions defined in \cite{dyn'kin} that play a crucial role in the proof of the approximate solution result.
\begin{defn}
For each $r>0$ we define:
\begin{equation*}
h_1(r)=\inf_{k\in\Z_+}m_kr^{k-1},
\end{equation*}
\begin{equation*}
h(r)=\inf_{k\in\Z_+}m_kr^{k}.
\end{equation*}
\end{defn}
\begin{rmk}
Note that for $r\geq 1$, we have $h_1(r) = 1$.
\end{rmk}

\begin{prop}\label{prop:h1h}
Let $n\in\Z_+$. 
There are constants $C_1,C_2,Q_1,Q_2>0$ such that
\[
\dfrac{1}{r^n}h_1(r)\leq C_1h_1(Q_1r),\quad\forall r> 0,
\]
\[
\dfrac{1}{r^n}h(r)\leq C_2h(Q_2r),\quad\forall r> 0.
\]
\end{prop}

\begin{proof}
Let $k\geq n$. 
\begin{align*}
\dfrac{1}{r^n}m_kr^{k-1}& 
\leq c^{k-1}m_{k-1}r^{k-n-1}\\
&\leq c^{(k-1)+(k-2)+\dots +(k-n)}m_{k-n}r^{k-n-1}\\
&=C_1m_{k-n}(Q_1r)^{k-n-1},
\end{align*}
where the constants $C_1$ and $Q_1$ only depend on $n$. 
Then:
\begin{align*}
\dfrac{1}{r^n}h(r)&=\dfrac{1}{r^n}\inf_{k\in\Z_+}m_kr^{k-1}\\
&\leq \inf_{k\geq n}m_kr^{k-n-1}\\
&\leq C_1\inf_{k-n\geq 0}m_{k-n}(Q_1r)^{k-n-1}\\
&=C_1h(Q_1r).
\end{align*}
The proof for the function $h_1$ is analogous. 
\end{proof}

\begin{defn}
For $r>0$ we define:
\begin{equation*}
\N(r)=\min\{n\,:\,h_1(r)=m_nr^{n-1}\}.
\end{equation*}
\end{defn}
\begin{prop}
There exists a subsequence $(m_{n_k})_{k=1}^\infty$ such that:
\[\N\!\left(\dfrac{m_{n_k}}{m_{n_k+1}}\right) = n_k.\]
\end{prop}
\begin{proof}
We first assume $m_n^2 < m_{n-1}m_{n+1}$ for all $n \in \Z_+ \setminus \{0\}$, then for each such $n$ we shall prove: 
\[
\N(r) = n, \qquad \dfrac{m_n}{m_{n+1}} \leq r < \dfrac{m_{n-1}}{m_n}.
\]
Let $k < n$ be a non-negative integer, we have:
\[
m_kr^{k-1} = \dfrac{m_k}{m_{k+1}} \dfrac{m_{k+1}}{m_{k+2}}\cdots \dfrac{m_{n-1}}{m_n}m_nr^{k-1} > m_nr^{n-1},
\]
by our assumption on $(m_n)_{n=0}^\infty$.
Thus $h_1(r) \leq m_nr^{n-1} < m_kr^{k-1}$, and $\N(r) \geq n$.
On the other hand, for each non-negative integer $j > n$ we have:
\[
m_nr^{n-1} = \dfrac{m_n}{m_{n+1}}\dfrac{m_{n+1}}{m_{n+2}} \cdots \dfrac{m_{j-1}}{m_j}m_jr^{n-1} \leq m_jr^{j-1}.
\]
Therefore $\N(r) = n$. 
In particular, we have $\N\!\left(m_n/m_{n+1}\right)=n$.
In the general case one has to take the least subsequence of $(m_n)$ which is strictly log-convex.
\end{proof}
\begin{cor}
The function $\N$ is a decreasing step function such that $\N(r)=0$ for every $r\geq 1$ and $\lim_{r\to 0}\N(r)=\infty$.
\end{cor}
\begin{lem}\label{lem:mkrk}
Let $r>0$. 
If $n\leq k\leq\N(r)$, then:
\[
m_kr^k \leq m_nr^n.
\]
\end{lem}

\begin{proof}
Let $n\leq k\leq \N(r)$. Condition b) implies that:
\[
m_k^{\N(r)-n}\leq m_n^{\N(r)-k}m_{\N(r)}^{k-n}.
\]
Thus:
\begin{align*}
\left(m_kr^k\right)^{\N(r)-n}&\leq m_n^{\N(r)-k}r^{\N(r)n+k-n-kn}\left(m_{\N(r)}r^{\N(r)-1}\right)^{k-n}\\
&\leq m_n^{\N(r)-k}r^{\N(r)n+k-n-kn}\left(m_nr^{n-1}\right)^{k-n}\\
&\leq m_n^{\N(r)-n}r^{\N(r)n-n^2}\\
&=\left(m_nr^n\right)^{\N(r)-n}.
\end{align*}
\end{proof}

\section{Approximate solutions for vector fields}\label{sec:approx}
\label{sec2}

We shall denote the coordinates on $\R^N \times \R$ and on $\C^M$ by $(x,t)=(x_1, \dots, x_N,t)$ and $\zeta = (\zeta_1, \dots, \zeta_M)$, respectively.
For this section, we fix $\Omega^\prime$, an open neighborhood of the origin in $\R^{N}$, and $\Omega^{\prime \prime}$, an open set in $\C^{M}$.
Let
\begin{equation}\label{eq:L}
L = \dfrac{\partial}{\partial t} + \sum_{i = 1}^N a_i(x,t, \zeta)\dfrac{\partial}{\partial x_i} + \sum_{j = 1}^M b_j(x,t, \zeta)\dfrac{\partial}{\partial \zeta_j},
\end{equation}
be a vector field in $\Omega^\prime \times \R \times \Omega^{\prime \prime}$ where $a_i, b_j$ are holomorphic in the variable $\zeta$ and of class $\mathcal{C}^1$ in $(x,t)$.
\begin{defn}
Let $u_0 \in \mathcal{C}^1(\Omega^\prime \times \Omega^{\prime\prime})$ be given.
An $(\mathcal{M},t)$-approximate solution of $L$ on $\Omega^\prime \times \R \times \Omega^{\prime \prime}$ with initial datum $u_0$ is a function $u \in \mathcal{C}^1(\Omega^\prime \times \R \times \Omega^{\prime \prime})$ with the following properties: 
\begin{enumerate}
\item For $(x,\zeta) \in \Omega^\prime \times \Omega^{\prime \prime}$ we have $u(x, 0, \zeta) = u_0(x,\zeta)$;
\item For every compact set $K\Subset\Omega^\prime\times\Omega^{\prime\prime}$ there are constants $A,\gamma,\delta > 0$ such that:
\[
\sup_{(x, \zeta) \in K}|Lu(x,t,\zeta)|\leq Ah(\gamma|t|),\quad 0<|t|\leq \delta.
\]
\end{enumerate}
\end{defn}
Condition $(2)$ in the definition above is equivalent to: for every compact set $K\Subset\Omega^\prime\times\Omega^{\prime\prime}$ there are positive constants $A,\gamma,\delta$ such that
\begin{equation*}
\sup_{(x,\zeta) \in K}\left|Lu(x,t,\zeta)\right| \leq A^{k+1}\frac{M_k}{k!}(\gamma |t|)^k,\quad 0 < |t| \leq \delta.
\end{equation*}
In this section, we shall prove that there exists an $(\M,t)$-approximate solution $u$ of $L$ for every $u_0\in\Cm(\Omega^{\prime}\times\Omega^{\prime\prime})$ as initial datum when the coefficients of $L$ are functions of class $\Cm$ in $(x,t)$. 
Let $\A$ be the subspace of $\Ci(\Omega^\prime \times \Omega^{\prime \prime})$ consisting of all functions that are holomorphic with respect to $\zeta$ and of class $\Cm$ in the variable $x$.
First we shall assume that $a_i, b_j \in \A$ for the vector field \eqref{eq:L} (thus the coefficients of $L$ do not depend on $t$, the general case follows from this particular one) and denote by $\A[[t]]$ the space of formal power series in the variable $t$ with coefficients on $\A$. 
Then the vector field \eqref{eq:L} is an endomorphism of $\A[[t]]$.
Let $f \in \A$ be given and let $u^\sharp(x,\zeta,t) = \sum_{k = 0}^\infty u_k(x,\zeta)t^k$ be a formal solution of the problem:
\[
\begin{cases}
L u^\sharp(x,\zeta,t) = 0,\\
u^\sharp(x,\zeta,0) = f(x,\zeta).
\end{cases}
\]
In fact, we have:
\begin{align*}
u_0(x, \zeta) &= f(x, \zeta),\\
u_k(x, \zeta) &= -\dfrac{1}{k}\left\{\sum_{i=1}^Na_i(x,\zeta)\dfrac{\partial u_{k-1}}{\partial x_i}(x,\zeta)+\sum_{j=1}^Mb_j(x,\zeta)\dfrac{\partial u_{k-1}}{\partial \zeta_i}(x,\zeta)\right\},
\end{align*}
for each $k \in \Z_+ \setminus \{0\}$ and each $(x, \zeta) \in \Omega^\prime \times \Omega^{\prime \prime}$.
\begin{prop} \label{prop:derivuk}
For each compact set $K \subset \Omega^\prime \times \Omega^{\prime \prime}$ there exists $C = C_K > 0$ such that:
\begin{equation*}
\sup_{(x,\zeta)\in K}\left|\partial^\alpha_x\partial^\beta_\zeta u_k(x,\zeta)\right|\leq \dfrac{C^{1+|\alpha|+|\beta|+k}M_{|\alpha|+k}\beta!}{k!},
\end{equation*}
for all $\alpha\in\Z_+^N$ and $\beta\in\Z_+^M$.
\end{prop}
For a proof of the above proposition, see Lemma $4.1$ in \cite{petronilho:09} and Lemma $18$ in \cite{hoepfner:10}, where the Gevrey case and the strongly non-quasi-analytic case, respectively, are proved; the proofs also hold in our case for they are based only on the log-convexity property.
We save the the symbol $C$ for the constant in Proposition \ref{prop:derivuk}.
\begin{defn}
For $n\in\Z_+$ define $\T^n:\A[[t]]\longrightarrow\A[[t]]$ by
\[
\T^n\left[\sum_{k=0}^\infty s_k(x,\zeta)t^k\right]=\sum_{k=0}^ns_k(x,\zeta)t^k,
\]
where $\sum_{k=0}^\infty s_k(x,\zeta)t^k\in\A[[t]]$.
\end{defn}
\begin{prop}\label{prop:estLTn}
For each compact set $K \subset \Omega^\prime \times \Omega^{\prime \prime}$ there exists $B = B_K > 0$ such that:
\[
\sup_{(x, \zeta) \in K}\left|L(\T^nu^\sharp)(x, \zeta, t)\right| \leq B^{n+1}m_n|t|^n.
\]
\end{prop}
\begin{proof}
We have the following identity of formal power series:
\begin{align*}
L\left(\T^nu^\sharp\right)(x,\zeta,t) &= L\left( \sum_{k=0}^n u_k(x,\zeta)t^k\right)\\
&=L\left(u^\sharp(x,\zeta,t) - \sum_{k = n+1}^\infty u_k(x, \zeta)t^k\right) \\
& = - L\left(\sum_{k = n+1}^\infty u_k(x, \zeta)t^k\right)\\
&=\left[(n+1)u_{n+1}(x, \zeta)\right]t^n+Q(x,\zeta,t),
\end{align*}
where $Q(x,\zeta,t)\in t^{n+1}\A[[t]]$.
But since the left-hand side of the previous equation is a polynomial in the variable $t$ of degree $n$, we have that $L\left(\T^nu^\sharp\right)(x,\zeta,t)=\left[(n+1)u_{n+1}(x, \zeta)\right]t^n$.
Now the result follows from Proposition \ref{prop:derivuk} combined with property (c) of the regular Denjoy-Carleman classes definition.
\end{proof}
Now we can use the technique presented in \cite{dyn'kin} to define an $(\M, t)$-approximate solution $u$ for the vector field \eqref{eq:L} with initial datum $f \in \A$. 
Let $\epsilon > 0$ be given and let $\psi \in \Ci_c(D_\epsilon(0))$ be a cutoff function such that $\psi \geq 0$, $\psi(z) = \psi(|z|)$ for all $z$, and
\[\int_\C \psi(z) \d z \wedge \d \bar{z} = 2/i.\]
Fix $U\Subset \Omega^\prime$ a neighborhood of the origin and $V\Subset\Omega^{\prime\prime}$ an open set. Now define for $x\in U$, $\zeta\in V$ and $|t|>0$
\[
u(x, t, \zeta) = \dfrac{i}{2t^2} \int_\C \psi\left(\dfrac{z-t}{|t|}\right)\sum_{k=0}^{\N((1+\epsilon)C|z|)}u_k(x, \zeta)z^k \d z \wedge \d \bar{z}.
\]
The function under the integral sign is measurable since $\N(r)$ is a step function, so $u$ is well defined. Differentiating under the integral sign we conclude that $u$ is holomorphic in $\zeta$.
Because of the choice of $\psi$ we have that $\lim_{t\to 0}u(x,t,\zeta)=u_0(x,\zeta)=f(x,\zeta)$. So we can set $u(x,0,\zeta)=f(x,\zeta)$.
In view of the symmetry of $\psi$, we have:
\[
\dfrac{i}{2t^2}\int_\C \psi\left( \dfrac{z - t}{|t|} \right) P(z) \d z \wedge \d \bar{z} = P(t),
\]
for every polynomial $P(z)$, in fact:
\begin{align*}
\dfrac{i}{2t^2}\int_\C \psi\left( \dfrac{z - t}{|t|} \right) P(z) \d z \wedge \d \bar{z} 
& =  \dfrac{i}{2}\int_\C\psi(w)P(|t|w+t) \d w \wedge \d \bar{w} \\
& = P(t) + \dfrac{i}{2}\int_\C\psi(w)Q(t, |t|w) \d w \wedge \d \bar{w},
\end{align*}
where $Q(t, z)$ is a polynomial such that $Q(t,0) = 0$, hence
\[\int_\C \psi(w)Q(t,|t|w) \d w \wedge \d \bar{w} = 0.\]
Therefore we have
\begin{align*}
Lu(x,t, \zeta) &= L\left[\sum_{k=0}^nu_k(x, \zeta)t^k + \dfrac{i}{2t^2}\int_\C \psi\left(\dfrac{z-t}{|t|}\right)\sum_{k=n+1}^{\N((1+\epsilon)C|z|)}u_k(x, \zeta)z^k \d z \wedge \d \bar{z}\right] \\
&=L\left(\T^nu^\sharp\right)(x, \zeta, t)+\dfrac{i}{2}\int_\C L\left[\dfrac{1}{t^2}\psi\left(\dfrac{z-t}{|t|}\right)\sum_{k=n+1}^{\N((1+\epsilon)C|z|)}u_k(x, \zeta)z^k \right]\d z \wedge \d \bar{z}\\
&=L\left(\T^nu^\sharp\right)(x, \zeta, t) + \dfrac{i}{2}\int_\C L\left[\dfrac{1}{t^2}\psi\left(\dfrac{z-t}{|t|}\right)\right]\sum_{k=n+1}^{\N((1+\epsilon)C|z|)}u_k(x, \zeta)z^k \d z \wedge \d \bar{z}\\
&+\dfrac{i}{2}\int_\C \dfrac{1}{t^2}\psi\left(\dfrac{z-t}{|t|}\right)\sum_{k=n+1}^{\N((1+\epsilon)C|z|)}L\left[u_k(x, \zeta)\right]z^k \d z \wedge \d \bar{z}
\end{align*}
By simple computations one can show that
\[
\left|L\left[\dfrac{1}{t^2}\psi\left(\dfrac{z-t}{|t|}\right)\right]\right|\leq \dfrac{C_1}{|t|^4},
\]
for some positive constant $C_1$. 
Since $0 \leq \psi \leq 1/(\pi\epsilon^2)$, we have:
\begin{align*}
|Lu(x,t,\zeta)| &\leq \left|L\left(\T^nu^\sharp\right)(x, \zeta, t)\right|+\dfrac{1}{2}\int_\C \left|L\left[\dfrac{1}{t^2}\psi\left(\dfrac{z-t}{|t|}\right)\right]\right|\sum_{k=n+1}^{\N((1+\epsilon)C|z|)}|u_k(x, \zeta)||z|^k |\d z \wedge \d \bar{z}|\\
&+\dfrac{1}{2}\int_\C \left|\dfrac{1}{t^2}\psi\left(\dfrac{z-t}{|t|}\right)\right|\sum_{k=n+1}^{\N((1+\epsilon)C|z|)}\left|Lu_k(x, \zeta)\right||z|^k |\d z \wedge \d \bar{z}|\\
&\leq \left|L\left(\T^nu^\sharp\right)(x, \zeta, t)\right|+\dfrac{C_1}{2|t|^4}\int_{|z-t|\leq |t|\epsilon}\sum_{k=n+1}^{\N((1+\epsilon)C|z|)}|u_k(x, \zeta)||z|^k |\d z \wedge \d \bar{z}|\\
&+\dfrac{1}{2\pi\epsilon^2}\int_{|z-t|\leq |t|\epsilon}\sum_{k=n+1}^{\N((1+\epsilon)C|z|)}\left|Lu_k(x, \zeta)\right||z|^k |\d z \wedge \d \bar{z}|.
\end{align*}
Now we fix $n=\N\left((1+\epsilon)^2C|t|\right)-1$. Note that $n$ must be positive, so from now on we shall assume $|t|\leq 1/(1+\epsilon)^2C=\delta$. 
Applying Lemma \ref{lem:mkrk} we can estimate: 
\[
\dfrac{M_k}{k!}\left((1+\epsilon)C|z|\right)^k \leq \dfrac{M_{n+1}}{(n+1)!} \left((1+\epsilon)C|z|\right)^{n+1},
\]
for $n < k \leq \N((1+\epsilon)C|z|)$.
Therefore, by Proposition \ref{prop:derivuk} and using $|z - t| \leq \epsilon|t|$, we have:
\begin{align*}
\sum_{k=n+1}^{\N((1+\epsilon)C|z|)}|u_k(x, \zeta)||z|^k & \leq \sum_{k=n+1}^{\N((1+\epsilon)C|z|)}C\dfrac{M_k}{k!}\left((1+\epsilon)C|z|\right)^k \dfrac{1}{(1+\epsilon)^k} \\
& \leq C_2 M_{n+1} \dfrac{C^{n+1}(1+\epsilon)^{2(n+1)}|t|^{n+1}}{(n+1)!} \\
& = C_2 (1+\epsilon)^2|t| h_1\!\left((1+\epsilon)^2C|t|\right),
\end{align*}
where this last equality follows from our choice of $n$.
Analogously, we have:
\begin{align*}
\sum_{k=n+1}^{\N((1+\epsilon)C|z|)}|Lu_k(x, \zeta)||z|^k & \leq C_3\sum_{k=n+1}^{\N((1+\epsilon)C|z|)}C^2\dfrac{M_{k+1}}{(k+1)!}\left((1+\epsilon)C|z|\right)^k \dfrac{1}{(1+\epsilon)^k}\\
& + C_4\sum_{k=n+1}^{\N((1+\epsilon)C|z|)}C\dfrac{M_k}{k!}\left((1+\epsilon)C|z|\right)^k \dfrac{1}{(1+\epsilon)^k}\\
& \leq C_5 M_{n+1} \dfrac{(1+\epsilon)^{2(n+1)}C^n|t|^{n+1}}{(n+1)!} \\
& = C_5 (1+\epsilon)^2|t| h_1\!\left((1+\epsilon)^2C|t|\right).
\end{align*}
By Proposition \ref{prop:estLTn}, we can also estimate the remaining term: 
\begin{align*}
\left|L\left(\T^nu^\sharp\right)(x, \zeta, t)\right| &\leq B^{n+1}m_n|t|^n\\
&\leq Bm_{n+1}(B|t|)^n \\ 
&= C_7h_1\!\left((1+\epsilon)^2C|t|\right).
\end{align*}
Summing up these three estimates and applying Proposition \ref{prop:h1h} we obtain
\begin{equation*}
|Lu(x,t,\zeta)|\leq C_8h(Q|t|),\qquad (x,\zeta)\in U\times V,\quad 0<|t|\leq \delta.
\end{equation*}
We claim that $u$ is a $\Ci$-function. 
We just have to check if $u$ is of class $\Ci$ at $\{(x,0,\zeta)\}$. 
For $n>0$ we have:
\begin{multline*}
\frac{1}{|t|^n}\left|\partial_x^\alpha \partial_\zeta^\beta u(x,t,\zeta) - \sum_{k=0}^{n}\partial_x^\alpha \partial_\zeta^\beta u_k(x,\zeta)t^k\right| \leq \\
\leq \dfrac{1}{2|t|^n}\int_\C \left|\dfrac{1}{t^2}\psi\left(\dfrac{z-t}{|t|}\right)\right|\sum_{k=n+1}^{\N((1+\epsilon)C|z|)}\left|\partial_x^\alpha \partial_\zeta^\beta u_k(x, \zeta)\right||z|^k |\d z \wedge \d \bar{z}|\\
=\dfrac{1}{2|t|^n}\int_\C|\psi(w)|\sum_{k=n+1}^{\N((1+\epsilon)C||t|w+t|)}\left|\partial_x^\alpha \partial_\zeta^\beta u_k(x, \zeta)\right|\left||t|w+t\right|^k|\d w \wedge \d \bar{w}|
\longrightarrow 0,
\end{multline*}
\noindent when $t\to 0$.
We proved the following theorem:
\begin{thm}\label{thm:sol_approx}
Let $\Omega = \Omega^\prime \times \R \times \Omega^{\prime \prime} \subset \R^N \times \R \times \C^M$ be an open set, where $\Omega^\prime \subset \R^{N}$ is an open neighborhood of the origin.
Let: 
\[
L = \dfrac{\partial}{\partial t} + \sum_{i = 1}^N a_i(x, \zeta)\dfrac{\partial}{\partial x_i} + \sum_{j = 1}^M b_j(x, \zeta)\dfrac{\partial}{\partial \zeta_j},
\]
be a vector field defined on $\Omega$, where $a_i, b_j \in\Ci(\Omega^\prime\times\Omega^{\prime\prime})$ are functions of class $\Cm$ with respect to $x$ and holomorphic in the variable $\zeta$. 
Let $f \in\Ci(\Omega^\prime\times\Omega^{\prime\prime})$ be a function of class $\Cm$ with respect to $x$ and holomorphic in the variable $\zeta$. 
Then for every open neighborhood of the origin $U \Subset \Omega^\prime$ and every open set $V \Subset \Omega^{\prime\prime}$,
there are a $\Ci$-function $u = u(x,t,\zeta)$ defined on $U \times\R \times V$ and holomorphic in $\zeta$ and constants $A, Q, \delta > 0$ such that:
\[
\begin{cases}
\left|Lu(x,t,\zeta)\right| \leq Ah(Q|t|), &(x, t, \zeta) \in U \times (-\delta,\delta) \times V,\\
u(x,0,\zeta) = f(x,\zeta), &(x,\zeta) \in U \times V.
\end{cases}
\]
\textit{i.e.}, the function $u$ is an $(\mathcal{M},t)$-approximate solution of $L$ on $U\times \R\times V$ with initial datum $f$.
\end{thm}
In Theorem \ref{thm:sol_approx} we assumed that the coefficients of $L$ do not depend on $t$, however one can obtain the general case from it:
\begin{thm} \label{thm:approx_sol}
Let $\Omega = \Omega^\prime \times I \times \Omega^{\prime \prime} \subset \R^N \times \R \times \C^M$, where $\Omega^\prime \subset \R^N$ is an open neighborhood of the origin and $\Omega^{\prime\prime}\subset\C^M$ is an open set.
Let 
\[
L = \dfrac{\partial}{\partial t} + \sum_{i = 1}^N a_i(x,t, \zeta)\dfrac{\partial}{\partial x_i} + \sum_{j = 1}^M b_j(x,t, \zeta)\dfrac{\partial}{\partial \zeta_j},
\]
be a vector field defined on $\Omega$, where $a_i$, $b_j$ are functions of class $\Cm$ with respect to the variables $(x,t)$ and holomorphic in the variable $\zeta$. Let $f\in\Ci(\Omega^\prime\times\Omega^{\prime\prime})$ holomorphic in $\zeta$ and $\Cm$ in $x$. 
Then for every open neighborhood of the origin $U \Subset \Omega^\prime$ and every open neighborhood of the origin $V \Subset \Omega^{\prime\prime}$, 
there are a $\Ci$-function $u = u(x,t,\zeta)$ defined on $U \times\R \times V$ and holomorphic in $\zeta$ and constants $A, Q, \delta > 0$ such that:
\[
\begin{cases}
\left|Lu(x,t,\zeta)\right| \leq Ah(Q|t|), &(x, t, \zeta) \in U \times (-\delta,\delta) \times V,\\
u(x,0,\zeta) = f(x,\zeta), &(x,\zeta) \in U \times V.
\end{cases}
\]
\textit{i.e.}, the function $u$ is an $(\mathcal{M},t)$-approximate solution of $L$ on $U\times \R\times V$ with initial datum $f$.
\end{thm}

\begin{proof}
Consider the vector field $\widetilde{L}$ in $\Omega\times \R$ defined by 

\begin{equation*}
\widetilde{L}=\partial_s + L,
\end{equation*}

\noindent and consider the function $\widetilde{f}(x,t,\zeta)=f(x,\zeta)$. Let $U \Subset \Omega^\prime$, $V\Subset\Omega^{\prime\prime}$ both neighborhoods of the origin and $r>0$ such that $(-r,r)\Subset I$. By Theorem \ref{thm:sol_approx} there exists a function $\widetilde{u}\in\Ci(U\times (-r,r)\times \R\times V)$ and constants $A,Q,\delta>0$ such that $U\times(-\delta,\delta)\times V\Subset \Omega\times \R$, $\widetilde{u}(x,t,0,\zeta)=\widetilde{f}(x,t,\zeta)$, for every $(x,t,\zeta)\in U\times (-r,r)\times V$, and 
\begin{equation} \label{eq:est_sol_approx_ger_dem}
|\widetilde{L}\widetilde{u}(x,t,s,\zeta)|\leq Ah(Q|s|), \quad (x,t,s,\zeta)\in U\times (-r,r)\times(-\delta,\delta)\times V.
\end{equation}
We shall assume $\delta < r$. 
Set $F(x,t,\zeta)=\widetilde{u}(x,t,t,\zeta)$ for $x\in U$, $\zeta\in V$ and $|t|<\delta$. 
We have
\begin{align*}
Lu(x,t,\zeta)&=L(\widetilde{u}(x,t,t,\zeta))\\
&=\widetilde{L}\widetilde{u}(x,t,t,\zeta).
\end{align*}
Therefore the desired estimate follows from \eqref{eq:est_sol_approx_ger_dem}.
\end{proof}

\section{Nonlinear PDEs}
\label{sec:nonlinear}

The following lemmas are the $\mathcal{C}^\mathcal{M}$-counterparts of results found in \cite{asano:95}.
We shall denote the coordinates in $\R^{N+1} = \R^N \times \R$ by $(x,t) = (x_1, \dots, x_N,t)$.

\begin{lem}\label{lem:L1}
Let $\Omega \subset \R^{N+1}$ be an open neighborhood of the origin.
Let
\[L = \dfrac{\partial}{\partial t} + \sum_{j=1}^Na_j(x,t)\dfrac{\partial}{\partial x_j},\]
be a vector field in $\Omega$ with coefficients in $\mathcal{C}^1(\Omega)$.
For each $1 \leq j \leq N$, suppose that there exists $Z_j \in \mathcal{C}^1(\Omega)$ an $(\mathcal{M}, t)$-approximate solution of $L$ with initial condition $Z_j(x, 0) = x_j$.
Then there exists a vector field
\[
L_1 = \dfrac{\partial}{\partial t} + \sum_{j=1}^Nb_j(x,t)\dfrac{\partial}{\partial x_j},
\]
defined on an open neighborhood of the origin $\Omega_1 \subset \Omega$ and with coefficients in $\mathcal{C}^1(\Omega_1)$ such that: 
\begin{enumerate}
\item For each $1 \leq j \leq N$ we have:
\begin{align*}
L_1(Z_j) &= 0, \\
a_j(x, 0) &= b_j(x, 0);
\end{align*}
\item Every $(\M, t)$-approximate solution of $L$ is an $(\M,t)$-approximate solution of $L_1$.
\end{enumerate}
\end{lem}

For a proof of Lemma \ref{lem:L1} see Section 2 of \cite{asano:95}, pp. 3010--3011.

\begin{lem}\label{lem:FBI}
Let $\Omega \subset \R^{N+1}$ be an open neighborhood of the origin.
Let
\[
L = \dfrac{\partial}{\partial t} + \sum_{j=1}^Na_j(x,t)\dfrac{\partial}{\partial x_j},
\]
be a vector field in $\Omega$ where $a_j \in \mathcal{C}^1(\Omega)$, $1 \leq j \leq N$.
Suppose that there exists $Z_j \in \mathcal{C}^1(\Omega)$ an $(\mathcal{M}, t)$-approximate solution of $L$ with initial condition $Z_j(x, 0) = x_j$
Let $\xi_0 \in \R^N \setminus \{0\}$ be such that $\Im\,a(0) \cdot \xi_0 < 0$.
Then there exists an open cone $\Gamma \subset \R^N\setminus \{0\}$, an open neighborhood of the origin $U \Subset \R^N$ , a cutoff function $\chi \in \Ci_c(\R^N)$, with $\chi = 1$ on $U$,  and constants $A > 0$ such that $\xi_0 \in \Gamma$ and
\begin{equation}
\left|\F\!\left[\chi \Psi_0\right](x,\xi)\right| \leq \dfrac{A^{k+1}M_k}{|\xi|^k}, \quad (x,\xi) \in V \times \Gamma, \quad k \in \Z_+,
\end{equation}
where $V \subset U$ is an open neighborhood of the origin, and $\Psi_0(x) = \Psi(x,0)$ is the trace of any $(\mathcal{M},t)$-approximate solution of $L$.
\end{lem}

\begin{proof}
Let $\xi_0 \in \R^N$ be such that $\Im \, a(0) \cdot \xi < 0$.
We apply Lemma \ref{lem:L1} and obtain an open neighborhood of the origin $\Omega_1 \subset \Omega$ and a vector field $L_1$ on $\Omega_1$.
We have $\d(H \d Z) = (L_1H)\d t \wedge \d Z$, for every function $H \in \mathcal{C}^1(\Omega_1)$.
Let $B \subset \R^N$ be an open ball around the origin and $I \subset \R$ an open interval around zero such that $B \times I \Subset \Omega_1$.
Choose a cutoff function $\chi \in \Ci_c(B)$ such that $\chi \equiv 1$ on a neighborhood $U \subset B$ of the origin and $0 \leq \chi \leq 1$.
Thus, if we fix an approximate solution $\Psi$ of $L$ and choose $H = H_{y,\xi}$ by the formula\footnote{where $y, \xi$ are parameters.}:
\[
H(x,t) = e^{i \xi \cdot (y - Z(x,t)) - |\xi|\<y - Z(x,t)\>^2}\chi(x)\Psi(x,t), \quad (x,t) \in \Omega_1, 
\]
then, for each positive $\lambda \in I$ we can apply Stokes' theorem and get:
\begin{align*}
\int_{B \times [0,\lambda]}(L_1H) \d t \wedge \d Z &=  
\int_{B \times [0,\lambda]} \chi(x)e^{i \xi \cdot (y - Z(x,t)) - |\xi|\<y - Z(x,t)\>^2}L_1 \left[\Psi(x,t)\right]\d t \wedge \d Z\\
&\quad + 
\int_{B \times [0,\lambda]}e^{i \xi \cdot (y - Z(x,t)) - |\xi|\<y - Z(x,t)\>^2}\Psi(x,t)L_1\left[\chi(x)\right] \d t \wedge \d Z \\
&= \int_{\partial(B \times [0,\lambda])}H \d Z \\
&=\int_{x \in B}e^{i \xi \cdot (y - Z(x,\lambda)) - |\xi|\<y - Z(x,\lambda)\>^2}\chi(x)\Psi(x,\lambda) \d_{x}Z(x,\lambda) \\
& \quad - 
\int_{x \in B}e^{i \xi \cdot (y - Z(x,0)) - |\xi|\<y - Z(x,0)\>^2}\chi(x)\Psi(x,0) \d_{x}Z(x,0), \\
\end{align*}
thus:
\begin{align} \label{est:FBI}
\nonumber |\F[\chi \Psi_0](y,\xi)| &\leq \int_{B \times [0,\lambda]}\chi(x)\left|e^{i \xi \cdot (y - Z(x,t)) - |\xi|\<y - Z(x,t)\>^2}L_1 \left[\Psi(x,t)\right]\right|\d t \wedge \d Z \\ 
&\quad + \int_{B \times [0,\lambda]}\left|e^{i \xi \cdot (y - Z(x,t)) - |\xi|\<y - Z(x,t)\>^2}\Psi(x,t)L_1\left[\chi(x)\right] \right|\d t \wedge \d Z \\ \nonumber
&\quad + \int_{x \in B}\left|e^{i \xi \cdot (y - Z(x,\lambda)) - |\xi|\<y - Z(x,\lambda)\>^2}\chi(x)\Psi(x,\lambda)\right| \d_{x}Z(x,\lambda).
\end{align}
Let $Q(x,t,y,\xi) = i \xi \cdot (y - Z(x,t)) - |\xi|\<y - Z(x,t)\>^2$, then as in \cite{asano:95} there exists an open cone $\Gamma \subset \R^N$, with $\xi_0 \in \Gamma$, an open neighborhood of the origin $V \subset \R^N$, and constants $C_0, \delta > 0$ such that
\[
\Re\,Q(x,t,y,\xi) \leq -C_0t|\xi|/2,
\]
for all $x \in B$, $\xi \in \Gamma$, $y \in V$ and $0 < t < \delta$.
Taking $\delta \in I$ and $V \subset U$, we can estimate:
\begin{multline*}
|\xi|^k\int_{B \times [0,\delta]}\chi(x)\left|e^{i \xi \cdot (y - Z(x,t)) - |\xi|\<y - Z(x,t)\>^2}L_1 \left[\Psi(x,t)\right]\right|\d t \wedge \d Z \\
\leq \int_{B \times [0, \delta]}|\xi|^ke^{-C_0t|\xi|/2}C^k\dfrac{M_{k-1}}{(k-1)!}|t|^{k-1}\sup_{(x,t) \in B \times [0, \delta]}|\det Z_x(x,t)| \, \d t \, \d x \\
= C^k \m(B)\sup_{(x,t) \in B \times [0, \delta]}|\det Z_x(x,t)| \dfrac{M_{k-1}|\xi|^k}{(k-1)!}\int_0^\delta e^{-C_0t|\xi|/2}t^{k-1} \d t\\
\leq C^k \m(B)\sup_{(x,t) \in B \times [0, \delta]}|\det Z_x(x,t)| \dfrac{M_{k-1}|\xi|^k}{(k-1)!}\int_0^\infty e^{-C_0t|\xi|/2}t^{k-1} \d t \\
= C^k \m(B)\sup_{(x,t) \in B \times [0, \delta]}|\det Z_x(x,t)| \dfrac{M_{k-1}|\xi|^k}{(C_0|\xi|/2)^{k}}\\
\leq {\left(\dfrac{2C}{C_0}\right)\!}^k\m(B)\sup_{(x,t) \in B \times [0, \delta]}|\det Z_x(x,t)| M_{k}, \quad \xi \in \Gamma, y \in V.
\end{multline*}
As in \cite{asano:95}, the remaining terms in \eqref{est:FBI} have exponential decay in some conic neighborhood of the origin.
\end{proof}

\begin{thm}\label{thm:sol_approx_linearized}
Let $\Omega = \Omega^\prime \times I \subset \R^{N} \times \R$ be an open neighborhood of the origin and let $\Omega^{\prime\prime} \subset \C^{N+1}$ be an open set.
Let $u \in \mathcal{C}^2(\Omega)$ be a solution of the nonlinear PDE:
\begin{equation}\label{eq:uteqf}
u_t = f(x,t,u,u_x),
\end{equation} 
where $f(x,t,\zeta_0, \zeta)$ is a function of class $\Cm$ with respect to $(x,t) \in \Omega$ and holomorphic with respect to $(\zeta_0,\zeta) \in \Omega^{\prime \prime}$. 
Let $L^u$ be the linearized operator:
\begin{equation}
L^u = \dfrac{\partial}{\partial t} - \sum_{j=1}^N \dfrac{\partial f}{\partial \zeta_j}(x,t,u,u_x)\dfrac{\partial}{\partial x_j}.
\end{equation}
Then for each open set $U \Subset \Omega^\prime$ there exist $\mathcal{C}^1$-functions $Z_j(x,t)$ and $\Psi(x,t)$ that are $(\M, t)$-approximate solutions of $L^u$ on $U \times \R$ with initial data $x_j$ and $u_0 = u(\,\cdot\,, 0)$, respectively, $j = 1, \dots, N$.





\end{thm}

\begin{proof}
In this proof we follow closely the proof of the Theorem $4.1$ of \cite{asano:95}. 
Consider the vector field
\begin{equation*}
\L=\frac{\partial}{\partial t}-\sum_{j=1}^N\dfrac{\partial f}{\partial\zeta_j}(x,t,\zeta_0,\zeta)\frac{\partial}{\partial x_j},
\end{equation*}
and the functions
\begin{align*}
h_0(x,t,\zeta_0,\zeta) &= f(x,t,\zeta_0,\zeta) - \sum_{j=1}^N\zeta_j\frac{\partial f}{\partial\zeta_j}(x,t,\zeta_0,\zeta)\\
h_i(x,t,\zeta_0,\zeta) &= \frac{\partial f}{\partial x_i}(x,t,\zeta_0,\zeta) + \zeta_i \frac{\partial f}{\partial\zeta_0}(x,t,\zeta_0,\zeta),\quad i = 1,\dots,N.
\end{align*}
This functions $h_j$ satisfies $h(x,u(x,t))=L^{u} w(x,t)$, where $w(x,t)=(u(x,t),u_x(x,t))$.
We can introduce now the holomorphic Hamiltonian 
\begin{equation*}
H=\L+h_0\frac{\partial}{\partial\zeta_0}+\sum_{j=1}^Nh_j\frac{\partial}{\partial\zeta_j}.
\end{equation*}
So it follows as in \cite{treves:92} that for every $\Phi(x,t,\zeta_0,\zeta)$ a $\Ci$-function,
\begin{equation}\label{eq:cadeia}
\L^w\Phi^w=\left(H\Phi\right)^w,
\end{equation}
\noindent and $\L^w=L^{v}$, with the notation $\Phi^w(x,t)=\Phi(x,t,w(x,t))$. 
Let $U\Subset \Omega^\prime$ be an open neighborhood of the origin and let $V\Subset \Omega^{\prime\prime}$ be an open neighborhood of $w(0,0)=(u(0,0),u_x(0,0))$ such that $w(x,t)\in V$ for all $(x,t)\in U$.
Applying Theorem \ref{thm:approx_sol} there exist functions $Z_j(x,t,\zeta_0,\zeta)$, $\Xi_k(x,t,\zeta_0,\zeta)$, $j=1,\dots,N$ and $k=0,1,\dots,N$, $\mathcal{C}^\infty$ in $(x,t)$ and holomorphic in $(\zeta_0,\zeta)$, $(\mathcal{M},t)$-approximate solutions of $H\Phi=0$ on $U\times\R\times V$ with initial conditions $Z_j(x,0,\zeta_0,\zeta)=x_j$, for $j=1,\dots,N$ and $\Xi_k(x,0,\zeta_0,\zeta)=\zeta_k$, for $k=0,1,\dots,N$. So there are constants $C_1,\rho,\delta>0$ such that
\begin{equation*}
\begin{cases}
\left|HZ_j(x,t,\zeta_0,\zeta)\right|\leq C_1h(\rho|t|),\quad\forall j=1,\dots,N,\\
\left|H\Xi_k(x,t,\zeta_0,\zeta)\right|\leq C_1h(\rho|t|),\quad\forall k=0,\dots,N,
\end{cases}
\end{equation*}
\noindent for $(x,\zeta_0,\zeta)\in U\times V$ and $|t|\leq \delta$.
The identity \eqref{eq:cadeia} implies that $Z_j^w(x,t)$ is an $(\M,t)$-approximate solution of $\L^w$ with initial condition $Z_j^w(x,0)=x_j$, for $j=1,\dots, N$. 
So it only remains to find an approximate solution of $\L^w$ with initial condition $u_0$.
Let $\widetilde{Z}(z,\overline{z},t,\zeta_0,\zeta)$ and $\widetilde\Xi(z,\overline{z},t,\zeta_0,\zeta)$ be $\mathcal{M}$-almost holomorphic extensions of $Z(x,t,\zeta_0,\zeta)$ and $\Xi(x,t,\zeta_0,\zeta)$ on $U\times\R\times V$, see \cite{dyn'kin}. Note that $\widetilde{Z}(z,\overline{z},t,\zeta_0,\zeta)$ and $\widetilde{\Xi}(z,\overline{z},t,\zeta_0,\zeta) $ are both holomorphic in $(\zeta_0,\zeta)$. Than there are positive constants $C_2,\gamma$ such that, shrinking $\delta$ if necessary, 

\begin{equation}\label{eq:ext_alm_hol}
\begin{dcases}
\left|\dfrac{\partial \widetilde{Z}_l}{\partial\overline{z}_j}(z,\overline{z},t,\zeta_0,\zeta)\right|\leq C_2h(\gamma|\Im \,z_j|),\quad\forall j,l=1,\dots,N\\
\left|\dfrac{\partial\widetilde{\Xi}_k}{\partial\overline{z}_j}(z,\overline{z},t,\zeta_0,\zeta)\right|\leq C_2h(\gamma|\Im \,z_j|),\quad\forall j=1,\dots,N,\,k=0,\dots,N,
\end{dcases}
\end{equation}
\noindent for $(z,\zeta_0,\zeta)\in \left(U+iB_\delta(0)\right)\times V$ and $|t|<\delta$.
Since 
\[\dfrac{\partial(\widetilde{Z},\overline{\widetilde{Z}},\widetilde{\Xi},\overline{\widetilde{\Xi}})}{\partial(z,\overline{z},\zeta_0,\overline{\zeta_0},\zeta,\overline{\zeta})}(z,\overline{z},t,\zeta_0,\zeta)\]
is non-singular if $t=0$ and $\Im\,z=0$, shrinking if necessary $U$, $V$ and $\delta$, one can use the implicit function theorem to solve
\[\left\{
\begin{array}{ll}
\widetilde{Z}(z,\overline{z},t,\zeta_0,\zeta)&=\widetilde{z}\\
\widetilde{\Xi}(z,\overline{z},t,\zeta_0,\zeta)&=\widetilde{\zeta},
\end{array}\right.\]
with respect to $(z,\zeta_0,\zeta)$ in $\left(U+iB_\delta(0)\right)\times V$. So there are two $\mathcal{C}^{\infty}$ functions  $P$ and $Q$ 
such that
\begin{equation*}
\begin{cases}
z=P(\widetilde{z},\overline{\widetilde{z}},t,\widetilde{\zeta},\overline{\widetilde{\zeta}})\\
(\zeta_0,\zeta)=Q(\widetilde{z},\overline{\widetilde{z}},t,\widetilde{\zeta},\overline{\widetilde{\zeta}}),
\end{cases}
\end{equation*}

\noindent with $P(0,0,\zeta_0,\zeta)=0$ and $Q(0,0,u(0),u_x(0))=(u(0),u_x(0))$. Combining this four equations we obtain

\begin{equation}\label{eq:sist_implic}
\begin{cases}
\widetilde{Z}(P(\widetilde{z},\overline{\widetilde{z}},t,\widetilde{\zeta}),\overline{P(\widetilde{z},\overline{\widetilde{z}},t,\widetilde{\zeta})},t,Q(\widetilde{z},\overline{\widetilde{z}},t,\widetilde{\zeta},\overline{\widetilde{\zeta}}))=\widetilde{z}\\
\widetilde{\Xi}(P(\widetilde{z},\overline{\widetilde{z}},t,\widetilde{\zeta}),\overline{P(\widetilde{z},\overline{\widetilde{z}},t,\widetilde{\zeta})},t,Q(\widetilde{z},\overline{\widetilde{z}},t,\widetilde{\zeta},\overline{\widetilde{\zeta}}))=\widetilde{\zeta}.
\end{cases}
\end{equation}

\noindent Differentiating the system \eqref{eq:sist_implic} with respect to $\overline{\widetilde{z}}$  we obtain 
\begin{align*}
&\dfrac{\partial(\widetilde{Z},\widetilde{\Xi})}{\partial(z,\zeta_0,\zeta)}(P(\widetilde{z},\overline{\widetilde{z}},t,\widetilde{\zeta},\overline{\widetilde{\zeta}}),\overline{P(\widetilde{z},\overline{\widetilde{z}},t,\widetilde{\zeta},\overline{\widetilde{\zeta}})},t,Q(\widetilde{z},\overline{\widetilde{z}},t,\widetilde{\zeta},\overline{\widetilde{\zeta}}))\dfrac{\partial(P,Q)}{\partial\overline{\widetilde{z}}}(\widetilde{z},\overline{\widetilde{z}},t,\widetilde{\zeta},\overline{\widetilde{\zeta}})\\
&+\dfrac{\partial(\widetilde{Z},\widetilde{\Xi})}{\partial(\overline{z},\overline{\zeta_0},\overline{\zeta})}(P(\widetilde{z},\overline{\widetilde{z}},t,\widetilde{\zeta},\overline{\widetilde{\zeta}}),\overline{P(\widetilde{z},\overline{\widetilde{z}},t,\widetilde{\zeta},\overline{\widetilde{\zeta}})},t,Q(\widetilde{z},\overline{\widetilde{z}},t,\widetilde{\zeta},\overline{\widetilde{\zeta}}))\dfrac{\partial(\overline{P},\overline{Q})}{\partial\overline{\widetilde{z}}}(\widetilde{z},\overline{\widetilde{z}},t,\widetilde{\zeta},\overline{\widetilde{\zeta}})=0.
\end{align*}
Let $A(z,s,\zeta_0,\zeta)$ be a generic entry of the matrix
\[\dfrac{\partial(\widetilde{Z},\widetilde{\Xi})}{\partial(\overline{z},\overline{\zeta_0},\overline{\zeta})}(z,\overline{z},t,\zeta_0,\zeta).\]

\noindent From the estimates \eqref{eq:ext_alm_hol} and that $\widetilde{Z}$ and $\widetilde{\Xi}$ are holomorphic in $(\zeta_0,\zeta)$ follows that
\[|A(z,t,\zeta_0\zeta)|\leq C_3 h(\gamma|\Im\,z|),\quad\forall (z,t,\zeta_0,\zeta)\in (U+iB_\delta(0))\times(-\delta,\delta)\times V,\]

\noindent for some positive constant $C_3$.
Since the (complex) matrix 
\[\dfrac{\partial(\widetilde{Z},\widetilde{\Xi})}{\partial(z,\zeta_0,\zeta)}(z,\overline{z},t,\zeta_0,\zeta)\]
is invertible for $\Im\,z=0$ and $t=0$, it follows that (shrinking $U,V$ and $\delta$ if necessary) 
\begin{equation}\label{est:delbarzQ0}
\left|\dfrac{\partial Q_0}{\partial\overline{\widetilde{z}}_j}(\widetilde{z},\overline{\widetilde{z}},t,\widetilde{\zeta})\right|\leq C_4h(\gamma|\Im \, P(\widetilde{z},\overline{\widetilde{z}},t,\widetilde{\zeta},\overline{\widetilde{\zeta}})|),\quad \forall j=1,\dots,N.
\end{equation}

\noindent for some positive constants $C_4$.
Analogously, differentiating the system \eqref{eq:ext_alm_hol} with respect to $\overline{\widetilde{\zeta}}$ and reasoning as before we have
\begin{equation}\label{est:delbarzetaQ0}
\left|\dfrac{\partial Q_0}{\partial\overline{\widetilde{\zeta}}_j}(\widetilde{z},\overline{\widetilde{z}},t,\widetilde{\zeta})\right|\leq C_5h(\gamma|\Im \,P(\widetilde{z},\overline{\widetilde{z}},t,\widetilde{\zeta},\overline{\widetilde{\zeta}})|),\quad \forall j=0,\dots,N,
\end{equation}
For some positive constants $C_5$.
Define the function $\Psi(z,\overline{z},t,\zeta_0,\zeta)$ for $(z,\overline{z},t,\zeta_0,\zeta)\in (U+iB_\delta(0))\times(-\delta,\delta)\times V$ by
\[\Psi(z,\overline{z},t,\zeta_0,\zeta,\overline{\zeta_0},\overline{\zeta})=Q_0\left(\widetilde{Z}(z,\overline{z},t,\zeta_0,\zeta),\overline{\widetilde{Z}(z,\overline{z},t,\zeta_0,\zeta)},0,\widetilde{\Xi}(z,\overline{z},t,\zeta_0,\zeta),\overline{\widetilde{\Xi}(z,\overline{z},t,\zeta_0,\zeta)}\right).\]
And by the definition of $\Psi$ we have
\begin{align*}
\Psi^w(x,0)&=\Psi(x,0,w(x,0))\\
&=\Psi(x,0,v(x,0),v_x(x,0))\\
&=Q_0(x,0,v(x,0),v_x(x,0))\\
&=v(x,0).
\end{align*}
Note that $H$ has no derivatives on $\Im \, z$, so $H\widetilde{Z}(x,t,\zeta_0,\zeta)=HZ(x,t,\zeta_0,\zeta)$ and the same happens for $H\widetilde{\Xi}$ at $\Im\,z =0$.
We have:
\[H\Psi=\sum_{j=1}^{N}\left(\dfrac{\partial Q_0}{\partial\widetilde{z}_j}H\widetilde{Z}_j+\dfrac{\partial Q_0}{\partial\overline{\widetilde{z}}_j}H\overline{\widetilde{Z}}_j\right)+\sum_{k=0}^{N}\left(\dfrac{\partial Q_0}{\partial\widetilde{\zeta}_k}H\widetilde{\Xi}_k+\dfrac{\partial Q_0}{\partial\overline{\widetilde{\zeta}}_k}H\overline{\widetilde{\Xi}}_k\right).
\]
and also
\begin{align*}
P(x,0,\zeta_0,\zeta,\overline\zeta_0,\overline{\zeta})&=P(\widetilde{Z}(x,0,\zeta_0,\zeta_0),\overline{\widetilde{Z}(x,0,\zeta_0,\zeta_0)},0,\widetilde{\Xi}(x,0,\zeta_0,\zeta),\overline{\widetilde{\Xi}(x,0,\zeta_0,\zeta)})\\
&=x,
\end{align*}
so 
\begin{equation*}
\Im \, P(\Re \, \widetilde{z},0,\widetilde\zeta,\overline{\widetilde\zeta}) = 0.
\end{equation*}
By the mean value inequality,
\begin{align*}
|\Im \,P(\widetilde{z},\overline{\widetilde{z}},0,\widetilde\zeta,\overline{\widetilde\zeta})|&=|\Im \,P(\widetilde{z},\overline{\widetilde{z}},0,\widetilde\zeta,\overline{\widetilde\zeta})-\Im \,P(\Re \, \widetilde{z},0,\widetilde\zeta,\overline{\widetilde\zeta})|\\
&\leq C_6|\Im\,\widetilde{z}|,
\end{align*}
For some positive constant $C_6$.
On the other hand, since 
\[\widetilde{Z}(x,0,\zeta_0,\zeta)=x,\]
there is $C_7>0$ such that 
\[|\Im\,\widetilde{Z}(x,t,\zeta_0,\zeta)\leq C_7|t|,\quad (x,t,\zeta_0,\zeta)\in U\times(-\delta,\delta)\times V.\]
Combining this two estimates with \eqref{est:delbarzQ0} and \eqref{est:delbarzetaQ0}, taking $C=\max C_j$ we obtain
\[\left|\dfrac{\partial Q_0}{\partial \overline{\widetilde{z}}_j}(\widetilde{Z}(x,t,\zeta_0,\zeta),\overline{\widetilde{Z}(x,t,\zeta_0,\zeta)},0,\widetilde{\Xi}(x,t,\zeta_0,\zeta))\right|\leq Ch(\gamma|t|),\quad (x,t,\zeta_0,\zeta)\in U\times(-\delta,\delta)\times V,\]
\noindent and 
\[\left|\dfrac{\partial Q_0}{\partial \overline{\widetilde{\zeta}}_j}(\widetilde{Z}(x,t,\zeta_0,\zeta),\overline{\widetilde{Z}(x,t,\zeta_0,\zeta)},0,\widetilde{\Xi}(x,t,\zeta_0,\zeta))\right|\leq Ch(\gamma|t|),\quad (x,t,\zeta_0,\zeta)\in U\times(-\delta,\delta)\times V.\]
Summing up we have
\[|H\Psi(x,t,\zeta_0,\zeta)|\leq Ch(\gamma|t|),\quad (x,t,\zeta_0,\zeta)\in U\times(-\delta,\delta)\times V.\]
So in view of equation \eqref{eq:cadeia} we have
\begin{equation*}
\begin{cases}
|\mathcal{L}^w\Psi^w(x,t)|\leq Ch(\gamma|t|),\quad (x,t)\in U\times (-\delta,\delta),\\
\Psi^w(x,0)=u(x,0), \quad x\in U.
\end{cases}
\end{equation*}
Thus we have constructed an $(\M,t)$-approximate solution of $\mathcal{L}^w$ with initial condition $u_0$.
\end{proof}

\begin{thm} \label{thm:WF_v0}
Let $\Omega = \Omega^\prime \times I \subset \R^{N} \times \R$ be an open neighborhood of the origin.
Let $v \in \mathcal{C}^2(\Omega)$ be a solution of the nonlinear PDE:
\begin{equation} \label{eq:vt}
v_t = g(x,v,v_x),
\end{equation} 
where $g(x, \zeta_0, \zeta)$ is a function of class $\Cm$ with respect to $x \in \Omega^\prime$ and holomorphic with respect to $(\zeta_0,\zeta) \in \C \times \C^N$.
Then:
\begin{equation}\label{eq:inclusao}
\left.\WF_\M(v_0)\right|_0\subset \big\{(0,\xi) \in \Omega^\prime \times \R^N \,:\,\Im \, b(0) \cdot \xi \geq 0 \big\},
\end{equation}
where $v_0 \in \mathcal{C}^2(\Omega^\prime)$ is given by $v_0(x) = v(x,0)$, $x \in \Omega^\prime$, and $b(x) = \nabla_\zeta g(x, v_0(x), v_{0x}(x))$.
\end{thm}

\begin{proof}
Just apply Lemma \ref{lem:FBI} with Theorem \ref{thm:sol_approx_linearized}.
\end{proof}

Applying a technique of Hanges-Treves presented  in \cite{treves:92}, we have the regularity theorem as a consequence of Theorem \ref{thm:WF_v0}:
\begin{thm}
Let $\Omega = \Omega^\prime \times I\times\Omega^{\prime\prime} \subset \R^{N} \times \R\times \C^N$, where $\Omega^\prime\times \R$ is an open neighborhood of the origin and $\Omega^{\prime\prime}$ is an open set.
Let $u \in \mathcal{C}^2(\Omega)$ be a solution of the nonlinear PDE:
\begin{equation}
u_t = f(x,t,u,u_x),
\end{equation} 
where $f(x,t,\zeta_0, \zeta)$ is a function of class $\Cm$ with respect to $(x,t) \in \Omega$ and holomorphic with respect to $(\zeta_0,\zeta) \in \C \times \C^N$.
Then:
\[
\WF_\M(u) \subset \Char(L^u),
\]
where $L^u$ is the linearized operator:
\begin{equation}
L^u = \dfrac{\partial}{\partial t} - \sum_{j=1}^N \dfrac{\partial f}{\partial \zeta_j}(x,t,u,u_x)\dfrac{\partial}{\partial x_j}.
\end{equation}
\end{thm}
For the convenience of the reader we present Hanges-Treves' argument.
We shall prove:
\[
\WF_\M(u)|_0 \subset \Char(L^u)|_0.
\]
The direction $(0; \xi, \tau) \in \Omega \times (\R^{N} \times \R)$ belongs to $\Char(L^u)$ if and only if:
\begin{equation} \label{eq:char}
\begin{cases}
\tau = - \Re \, a(0) \cdot \xi,\\
0 = \Im \, a(0) \cdot \xi,
\end{cases}
\end{equation}
where $a(x,t) = \nabla_{\zeta}f(x,t,u(x,t),u_x(x,t))$.
For each $\theta \in [0, 2\pi)$ one can see that $v(x,t,s) = u(x,t)$ is a $\mathcal{C}^2$-solution of the following nonlinear PDE:
\begin{equation} \label{eq:vs}
v_s = f^\theta(x,t,v,v_x,v_t),
\end{equation}
where $f^\theta(x,t,\zeta_0,\zeta,\zeta_{N+1}) = e^{-i\theta}(\zeta_{N+1} - f(x,t,\zeta_0,\zeta))$ and we are setting the coordinates in $\R^N \times \R \times \R$ as $(x,t,s) = (x_1, \dots,x_N, t, s)$ and the coordinates in $\C \times \C^N \times \C$ as $(\zeta_0, \zeta, \zeta_{N+1}) = (\zeta_0, \zeta_1, \dots, \zeta_N, \zeta_{N+1})$.
The corresponding linearized operator is: 
\[
L^\theta = \dfrac{\partial}{\partial s} - e^{-i\theta}L^u.
\]
The direction $(0;\xi,\tau,\sigma) \in (\Omega \times \R) \times (\R^N \times \R \times \R)$ belongs to $\Char(L^\theta)$ if and only if:
\begin{equation} \label{eq:chartheta}
\begin{cases}
\sigma = [(\cos\theta)\Re \, a(0) - (\sin\theta)\Im \, a(0)] \cdot \xi + (\cos\theta) \tau,\\
    0 = [(\cos\theta)\Im \, a(0) + (\sin\theta)\Re \, a(0)]\cdot \xi + (\sin\theta)\tau.
\end{cases}
\end{equation}
One can notice that the validity of \eqref{eq:char} is equivalent to the validity of the second equation on \eqref{eq:chartheta} for every $\theta \in [0,2\pi)$.
Now, let $(0; \xi_0, \tau_0) \notin \Char(L^u)$.
There exists $\theta \in [0, 2\pi)$ such that $[(\cos\theta)\Im \, a(p) + (\sin\theta)\Re \, a(p)]\cdot \xi_0 + (\sin\theta)\tau_0 \neq 0$. 
By choosing among $\theta$, $\theta + \pi$ and $\theta - \pi$, one can suppose $[(\cos\theta)\Im \, a(p) + (\sin\theta)\Re \, a(p)] \cdot \xi_0 + (\sin\theta)\tau_0 < 0$.
Applying Theorem \ref{thm:WF_v0} to the solution $v(x,t,s) = u(x,t)$ of \eqref{eq:vs} we conclude $(0;\xi_0,\tau_0) \notin \WF_\M(v_0)|_0 = \WF_\M(u)|_0$.

\def\cprime{$'$}


\begin{thebibliography}{10}

\bibitem{hoepfner:10}
Z. ~Adwan and G. ~Hoepfner.
\newblock Approximate solutions and micro-regularity in the {D}enjoy-{C}arleman classes.
\newblock {\em J. Differential Equations}, 249(9):2269--2286, 2010.

\bibitem{hoepfner:15}
Z. ~Adwan and G. ~Hoepfner.
\newblock Denjoy-{C}arleman classes: boundary values, approximate solutions and applications.
\newblock {\em J. Geom. Anal.}, 25(3):1720--1743, 2015.

\bibitem{asano:95}
C. H. ~Asano.
\newblock On the {$\Ci$} wave-front set of solutions of first-order nonlinear {PDE}s.
\newblock {\em Proc. Amer. Math. Soc.}, 123(10):3009--3019, 1995.

\bibitem{petronilho:09}
R. F. ~Barostichi and G. ~Petronilho.
\newblock Gevrey micro-regularity for solutions to first order nonlinear {PDE}.
\newblock {\em J. Differential Equations}, 247(6):1899--1914, 2009.

\bibitem{petronilho:11}
R. F. ~Barostichi and G. ~Petronilho.
\newblock Existence of {G}evrey approximate solutions for certain systems of linear vector fields applied to involutive systems of first-order nonlinear PDEs.
\newblock {\em J. Math. Anal. Appl.}, 382(1):248--260, 2011.

\bibitem{bierstone:04}
E. ~Bierstone and P. D. ~Milman.
\newblock Resolution of singularities in {D}enjoy-{C}arleman classes.
\newblock {\em Selecta Math. (N.S.)}, 10(1):1--28, 2004.

\bibitem{chemin:88}
J. Y. ~Chemin.
\newblock Calcul paradiff\'erentiel pr\'ecis\'e et applications \`a des \'equations aux deriv\'ees partielles non semilin\'eaires.
\newblock {\em Duke Math. J.}, 56:431--469, 1988.

\bibitem{bch}
S. Berhanu, P. D. Cordaro and J. Hounie.
\newblock {\em An introduction to involutive structures}.
\newblock New Mathematical Monographs, 6. Cambridge University Press, Cambridge, 2008.

\bibitem{dyn'kin}
E. M. ~Dyn'kin.
\newblock Pseudoanalytic extension of smooth functions. {T}he uniform scale.
\newblock {\em Amer. Math. Soc. Transl.}, 115(2):33--58, 1980.

\bibitem{treves:92}
N. ~Hanges and F. ~Treves.
\newblock On the analyticity of solutions of first-order nonlinear {PDE}.              
\newblock {\em Trans. Amer. Math. Soc.}, 331(2):627--638, 1992.


\bibitem{furdos}
 S.~F\"{u}rd\"{o}s.
 \newblock Ultradifferentiable {C}{R} {M}anofolds.
 \newblock {PhD. dissertation},
 \newblock {Universit\"{a}t {W}ien},
 \newblock {Vienna}, 2017.






  






\end{thebibliography}
\end{document}